\documentclass[secthm,seceqn,amsthm,ussrhead,12pt]{amsart}
\usepackage{amsmath,latexsym}
\usepackage[english]{babel}
\usepackage[psamsfonts]{amssymb}
\usepackage{times}
\usepackage{cite}

\usepackage[mathcal]{euscript}
\numberwithin{equation}{section} \textwidth=17.5cm
\newtheorem{theorem}{Theorem}[section]
\newtheorem{lemma}[theorem]{Lemma}

\newtheorem{corollary}[theorem]{Corollary}

\newtheorem{remark}[theorem]{Remark}

\numberwithin{equation}{section}

\begin{document}

\title[Self-adjoint cyclically
compact operators]{Self-adjoint cyclically compact operators and
their applications}

\author{Farrukh Mukhamedov  and Karimbergen Kudaybergenov}

\address[Farrukh Mukhamedov]{Department of Computational \& Theoretical Sciences\\
Faculty of Science, International Islamic University Malaysia\\
P.O. Box, 141, 25710, Kuantan\\
Pahang, Malaysia} \email{{\tt far75m@yandex.ru} {\tt
farrukh\_m@iium.edu.my}}

\address[Karimbergen Kudaybergenov]{Department of Mathematics, Karakalpak state
university,  230113 Nukus, Uzbekistan.} \email{karim2006@mail.ru}

\maketitle
\begin{abstract}
This paper is devoted to self-adjoint cyclically compact operators
on Hilbert--Kaplansky module over a ring of bounded measurable
functions. The spectral theorem for such a class of operators are
given. We apply this result to  partial integral equations on the
space with mixed norm of measurable functions and to compact
operators relative to von Neumann algebras. We will give a condition
of solvability of  partial integral equations with self-adjoint
kernel. Moreover, a general form of  compact operators relative to a
type I von Neumann algebra is given.

 \vskip 0.3cm \noindent

\\

\end{abstract} \maketitle

\bigskip

\section{Introduction}

\medskip

The modern structure theory of $AW^\ast$-modules originated with
the articles by I. Kaplansky \cite{Kap, Kap2} and nowadays this
theory  has many applications in the operator algebras.

One of the important instrument in the theory of operator algebras
is a different form spectral theorem. In \cite{Wri} it was proved an
important spectral theorem for Hilbert--Kaplansky modules. Another
important concept is the compactness property. Cyclically compact
sets and operators in lattice-normed spaces were introduced by
Kusraev (see \cite{Ku85}).  In \cite{Ku85} (see also \cite{KusDom})
a general form of cyclically compact operators in Hilbert--Kaplansky
module, as well as a variant of Fredholm alternative for cyclically
compact operators, are given. In \cite{GK01} it was proved that
every cyclically compact operator acting in Banach--Kantorovich
space over a ring measurable functions can be represented as a
measurable bundle of compact operators acting in Banach spaces. In
\cite{Alb2} it has been shown that the algebra of all locally
measurable operators with respect to a type I von Neumann algebra
can be represented as an algebra of all bounded module-linear
operators acting on a Hilbert–-Kaplanskiy module over the ring of
measurable function on a measure space. This result played  a
crucial role in the description of derivations on algebras of
locally measurable operators with respect to type I von Neumann
algebras and their subalgebras (see for example, \cite{Alb2, AK10,
AK1, BPS}).

It is well-known that one of the important notions in the theory of
operator algebras is  compact operators  relative to von Neumann
algebras with faithful semi-finite trace introduced by M. G. Sonis
1971 (see \cite{Son}). Later V. Kaftal \cite{Kaf} has shown that
Sonis's definition of compact operator relative to von Neumann
algebras with faithful semi-finite trace is viable for general von
Neumann algebras too and he obtained most of the classical
characterizations of compact operators.

In this paper we shall investigate self-adjoint cyclically compact
operators on Hilbert--Kaplansky module over a ring of bounded
measurable functions and their applications.

In Section 2 we give preliminaries from the theory of
Hilbert--Kaplansky module and give a general form of self-adjoint
cyclically compact operators on Hilbert--Kaplansky module over a
ring measurable functions.

In section 3 we  apply the above theorem  to partial integral
equations on the space with mixed norm of measurable functions. We
will give a condition of solvability of  partial integral
equations with self-adjoint kernel.

In section 4 we will give a general form  a self-adjoint compact
operator relative to  a type I von Neumann algebra.

\section{Self-adjoint cyclically compact operators on Hilbert--Kaplansky modules}

Let us recall some notions and results from the theory of
Hilbert--Kaplansky modules (see \cite{KusDom}).

Let $(\Omega, \Sigma, \mu)$  be a measurable space and suppose
that the measure $\mu$ has the direct sum property, i.e. there is
a family $\{\Omega_i\}_{i\in J}\subset\Sigma,$
$0<\mu(\Omega_i)<+\infty, i\in J$ such that for any $A \in \Sigma,
\mu(A)<+\infty$ there exist a countable subset $J_0\subset  J$ and
a set $B$ with zero measure such that $A = \bigcup\limits_{i\in
J_0}(A\cap \Omega_i)\cup B.$

We denote  by $L^{\infty}(\Omega)$ the algebra of all (equivalence
classes of) complex measurable bounded functions on $\Omega$ and
let $\nabla$ be the set of all idempotents of the algebra
$L^\infty(\Omega).$

Let  $X$ be a  unitary $L^\infty(\Omega)$-module. The mapping
$\langle\cdot,\cdot\rangle:X\times X\rightarrow L^\infty(\Omega)$
is a $L^\infty(\Omega)$-valued inner product, if for all $\xi,
\eta, \zeta\in  X$ and $a\in L^\infty(\Omega)$  the following are
satisfied:
\begin{itemize}
\item[(1)] $\langle \xi,\xi\rangle\geq 0, \langle \xi,\xi\rangle= 0 \Leftrightarrow \xi=0;$
\item[(2)] $\langle \xi,\eta\rangle=\overline{\langle \eta,\xi\rangle};$
\item[(3)] $\langle a\xi,\eta\rangle=a\langle \xi,\eta\rangle;$
\item[(4)] $\langle \xi+\eta,\zeta\rangle=\langle \xi,\zeta\rangle+
\langle \eta,\zeta\rangle.$
\end{itemize}

Using a $L^\infty(\Omega)$-valued inner product, we may introduce
the norm in $X$ by the formula
$$
\|\xi\|_{\infty}=\sqrt{\|\langle\xi,\xi\rangle\|_{L^{\infty}(\Omega)}},
$$
and the vector norm
$$
\|\xi\|=\sqrt{\langle\xi,\xi\rangle}.
$$
A Hilbert--Kaplansky module over $L^\infty(\Omega)$ is a unitary
module over $L^\infty(\Omega)$ such  that it is complete with
respect the norm $\|\cdot\|_\infty$ and the following two
properties are true:
\begin{itemize}
\item[(1)] let $\xi$ be an arbitrary element in $X,$ and let $\{\pi_i\}_{i\in I}$
be a partition of unity in $\nabla$ with $\pi_i\xi=0$ for all
$i\in I,$ then $\xi = 0;$
\item[(2)] let $\{\xi_i\}_{i\in I}$ be a norm-bounded family in $X,$ and let
$\{\pi_i\}_{i\in I}$ be a partition of unity in $\nabla,$ then
there exists an element $\xi\in X$ such that $\pi_i \xi =
\pi_i\xi_i$ for all $i\in I.$
\end{itemize}

An orthogonal basis in a Hilbert--Kaplansky module $X$ over
$L^\infty(\Omega)$ is a orthogonal set whose orthogonal complement
is $\{0\}.$ A Hilbert--Kaplansky module $X$ is said to be
$\alpha$-\textit{homogeneous,} if $\alpha$ is a cardinal and $X$
has a basis of cardinality $\alpha.$

A Hilbert--Kaplansky module $X$ is said to be
$\sigma$-\textit{finite-generated} if  there exists a partition of
unity $\{\pi_n\}_{n\in F}$  in $\nabla,$ where $F\subseteq
\mathbb{N},$ such that $\pi_n X$ is a $n$-homogeneous module over
$\pi_n L^\infty(\Omega)$  for all $n\in F.$

Let  $C$ be a subset in $X.$ Denote by  $\mbox{mix}(C)$ the set of
all vectors $\xi$ from $X$ for which there is a partition of unity
$\{\pi_i\}_{i\in I}$ in $\nabla$  such that $\pi_i \xi\in C$ for
all $i\in I,$ i.e.
$$
\mbox{mix}(C)=\left\{\xi\in X: \exists\, \pi_i\in \nabla,
\pi_i\pi_j=0, i\neq j, \bigvee\limits_{i\in I}\pi_i=\textbf{1},
\pi_i \xi\in C, i\in I\right\}.
$$
In other words $\mbox{mix}(C)$ is the set of all mixings obtained
by families $\{\xi_i\}_{i\in I}$ taken from $C.$

A subset $C$ is said to be \textit{cyclic} if $C=\mbox{mix}(C).$

For  a nonempty set $A$  by  $\nabla(A)$ denotes the set of all
partitions of unity in $\nabla$  with the index set $A,$ i.e.,
$$
\nabla(A)=\left\{\nu:A\rightarrow \nabla:(\forall\,
\alpha,\beta\in A)(\alpha\neq\beta\rightarrow
\nu(\alpha)\wedge\nu(\beta)=0)\wedge\bigvee\limits_{\alpha\in
A}\nu(\alpha)=\mathbf{1}\right\}.
$$
If $A$ is a partially ordered set then we can order the set
$\nabla(A)$ as:
$$
\nu\leq\mu\leftrightarrow (\forall\, \alpha,\beta\in
A)(\nu(\alpha)\wedge\nu(\beta)\neq0\rightarrow
\alpha\leq\beta)\,\, (\nu,\mu\in \nabla(A)).
$$
Then this relation is  a partial order in $\nabla(A),$ in
particularly, if  $A$ is directed upward or downward, then so does
$\nabla(A).$

Take any net $(\xi_\alpha)_{\alpha\in A}$  in $X.$  For each
$\nu\in \nabla(A)$  put $\xi_\nu=\mbox{mix}_{\alpha\in A}
\nu(\alpha)\xi_\alpha.$ If all the mixings exist then we have a
net $(\xi_\nu)_{\nu\in \nabla(A)}$ in $X.$ Every subnet of the net
$(\xi_\nu)_{\nu\in \nabla(A)}$  is called a cyclical subnet of the
original net $(\xi_\alpha)_{\alpha\in A}.$

Recall \cite{KusDom} that a subset $C\subset X$ is said to be
\textit{cyclically compact} if $C$ is cyclically  and any sequence
in $C$ has a cyclic subsequence that norm converges to some element
of $C.$ A subset in $X$ is called \textit{relatively cyclically
compact} if it is contained in a cyclically compact set.

An operator $T$ on $X$  is called
\textit{$L^\infty(\Omega)$-linear} if $T(a\xi+b\eta)=a
T(\xi)+bT(\eta)$ for all $a, b\in L^\infty(\Omega),$ $\xi, \eta\in
X.$

A $L^\infty(\Omega)$-linear operator $T$ on $X$  is called
\textit{cyclically compact}  if the image $T(C)$ of any bounded
subset $C\subset X$ is relatively cyclically compact.

For every $\xi, \eta\in X$ we define an operator $\xi\otimes \eta$
on $X$ by the rule
$$
(\xi\otimes\eta)(\zeta)=\langle\zeta, \eta\rangle\xi,\, \zeta\in
X.
$$

It is well-known  \cite[Theorem 8.5.6]{KusDom} that if $T$  is  a
cyclically compact operator on $X$ then there exists a partition of
unity $\{\pi_0, \pi_1, \ldots, \pi_k,\ldots, \pi_\infty\}$ in
$\nabla$ and  orthonormal system $\{\xi_{k, n}\}_{n=1}^{k},$
$\{\eta_{k, n}\}_{n=1}^k$ in $\pi_k X$ and a families $\{f_{k,
n}\}_{n=1}^k$ in $\pi_kL^\infty(\Omega),$ where $k=1, \ldots,
n,\ldots, \infty,$ such that the followings are true:
\begin{enumerate}
\item[(1)] $\pi_0 T=0;$
\item[(2)] $\pi_\infty f_{\infty, n}\downarrow 0;$
\item[(3)] the representation is valid
\begin{equation}\label{deco}
T=\pi_\infty\sum\limits_{n=1}^\infty f_{\infty, n}\, \xi_{\infty,
n} \otimes \eta_{\infty, n}+\sum\limits_{k=1}^\infty \pi_k
\sum\limits_{n=1}^k f_{k, n}\, \xi_{k,n} \otimes \eta_{k,n}.
\end{equation}
\end{enumerate}

The following is the main result of this section.

\begin{theorem}\label{MA}
If $T$  is  a self-adjoint cyclically compact operator on $X$ then
there are partition of unity $\{\pi_0, \pi_1, \ldots,
\pi_k,\ldots, \pi_\infty\}$ in $\nabla$ and  orthonormal families
$\{\xi_{k, n}\}_{n=1}^{k}$  in $\pi_k X$ and a families $\{f_{k,
n}\}_{n=1}^k$ in $\pi_kL^\infty(\Omega),$ where $k=1, \ldots,
n,\ldots, \infty,$ such that the following hold:
\begin{enumerate}
\item[(1)] $\pi_0 T=0;$
\item[(2)] $\pi_\infty |f_{\infty, n}|\downarrow 0;$
\item[(3)] the representation is valid
\begin{equation}\label{selfd}
T=\pi_\infty\sum\limits_{n=1}^\infty f_{\infty, n}\, \xi_{\infty,
n} \otimes \xi_{\infty, n}+\sum\limits_{k=1}^\infty \pi_k
\sum\limits_{n=1}^k f_{k, n}\, \xi_{k,n} \otimes \xi_{k,n}.
\end{equation}
\end{enumerate}
\end{theorem}

\begin{proof} Let $T$ be an operator of the form~\eqref{deco}.
Without lost of generality we can assume that
$\pi_\infty=\mathbf{1},$ and therefore the operator $T$ has the
form  $T=T^\ast=\sum\limits_{n=1}^\infty f_n\, \eta_n
\otimes\xi_n.$

Case 1. Let $T\geq 0.$ Since $T=|T|=\sqrt{T^\ast T}=\sqrt{T
T^\ast}$ and
$$
T T^\ast=\sum\limits_{n=1}^\infty f_n\, \xi_n \otimes
\eta_n\sum\limits_{m=1}^\infty f_m\, \eta_m \otimes
\xi_m=\sum\limits_{n,m=1}^\infty f_nf_m\, \langle\eta_m,
\eta_n\rangle\xi_m \otimes \xi_n =\sum\limits_{n=1}^\infty f_n^2\,
\xi_n \otimes \xi_n
$$
we get
$$
T=\sum\limits_{n=1}^\infty f_n\, \xi_n \otimes \xi_n.
$$

Case 2. Let $T$ be an arbitrary self-adjoint cyclically compact
operator. Consider its  representation in the form $T=T_+-T_{-},$
where $T_+, T_{-}\geq 0$ and $T_+T_{-}=0.$ By case 1 there are
orthonormal families  $\{\xi_n^+\}_{n\in \mathbb{N}},$
$\{\xi_n^-\}_{n\in \mathbb{N}}$ in $X$ and  families
$\{f_n^+\}_{n\in \mathbb{N}},$ $\{f_n^-\}_{n\in \mathbb{N}}$ in
$L^\infty(\Omega)$ such that $f_n^+\downarrow 0,$ $f_n^-\downarrow
0$ and the following representations are valid
$$
T_+=\sum\limits_{n=1}^\infty f_n^+\, \xi_n^+ \otimes \xi_n^+,
$$
$$
T_{-}=\sum\limits_{n=1}^\infty f_n^{-}\, \xi_n^{-} \otimes
\xi_n^{-}.
$$
Since $T_+T_{-}=0$ we get that the elements $\xi_i^{+}$ and
$\xi_j^{-}$ are mutually orthogonal for all $i,j\in \mathbb{N}.$

Now consider the following sequence in  $\nabla:$
$$
z_n=c((f_1^{-}-f_{n}^+)_{+}),\, n\geq 1,
$$
where $f_+$ is the positive part of self-adjoint element $f$ from
$L^\infty(\Omega)$ and $c(f)$ it's support, i.e. $c(f)$ is the
indicator function of the set $\{\omega\in\Omega: f(\omega)\neq
0\}.$
Set
$$
f_1^{(1)}=-z_1 f_1^-+z_1^\perp f_1^+,
$$
$$
f_n^{(1)}=z_{n-1} f_{n-1}^+ - (z_n-z_{n-1})f_{1}^-+z_n^\perp
f_n^{+},\, n\geq 2,
$$
$$
\xi_1^{(1)}=z_1 \xi_1^-+z_1^\perp \xi_1^+,
$$
$$
\xi_n^{(1)}=z_{n-1} \xi_{n-1}^++ (z_n-z_{n-1})\xi_{1}^-+z_n^\perp
\xi_n^{+},\, n\geq 2.
$$
Let us check that $|f_{n}^{(1)}|\downarrow 0.$ Firstly we shall
show that $|f_{1}^{(1)}|\geq |f_{2}^{(1)}|.$ Indeed,
$$
z_1|f_{1}^{(1)}|=z_1f_{1}^{-}\geq z_1 f_{1}^{+}=z_1|f_{2}^{(1)}|
$$
and
$$
z_1^\perp |f_{1}^{(1)}|=z_1^\perp f_{1}^{+}=(z_2-z_1)f_1^+
+z_2^\perp f_{1}^{+}\geq (z_2-z_1)f_1^- +z_2^\perp
f_{2}^{+}=z_1^\perp |f_{2}^{(1)}|.
$$

Now let $n>1.$ Then
$$
z_{n-1}|f_{n}^{(1)}|=z_{n-1}f_{n-1}^{+}\geq z_{n-1}
f_{n}^{+}=z_{n-1} z_n f_{n}^{+}=z_{n-1}|f_{n+1}^{(1)}|,
$$
$$
(z_{n}-z_{n-1})|f_{n}^{(1)}|=(z_n-z_{n-1})f_{1}^{-}\geq
(z_n-z_{n-1}) f_{n}^{+}=(z_n-z_{n-1})|f_{n+1}^{(1)}|
$$
and
$$
z_n^\perp |f_{n}^{(1)}|=z_n^\perp f_{n}^{+}=(z_{n+1}-z_n)f_n^+
+z_{n+1}^\perp f_{n}^{+}\geq (z_{n+1}-z_n)f_1^- +z_{n+1}^\perp
f_{n+1}^{+}=z_{n}^\perp |f_{n+1}^{(1)}|.
$$
So $|f_{n}^{(1)}|\geq |f_{n+1}^{(1)}|.$ Since $f_{n}^{+}\downarrow
0$ and $z_{n}^{\perp}\downarrow 0$ we get $|f_{n}^{(1)}|\downarrow
0.$

Direct computations show that
$$
\langle\xi_1^{(1)}, \xi_1^{(1)}\rangle=z_1+z_1^\perp=\mathbf{1},
$$
$$
\langle\xi_n^{(1)},
\xi_n^{(1)}\rangle=z_{n-1}+(z_n-z_{n-1})+z_n^\perp=\mathbf{1},\,n>1,
$$
$$
\langle\xi_n^{(1)}, \xi_m^{(1)}\rangle=0,\,n\neq m.
$$
This means that  $\{\xi_n^{(1)}\}$ is an orthonormal system.

Set
$$
T_1=\sum\limits_{n=1}^\infty f_n^{(1)}\, \xi_n^{(1)} \otimes
\xi_n^{(1)}-\sum\limits_{n=2}^\infty f_n^{-}\, \xi_n^{-} \otimes
\xi_n^{-}.
$$
Let us show that $T=T_1.$ It clear that $T(\xi_k^-)=T_1(\xi_k^-)$
for all $k>1.$ Further
$$
T_1(\xi_1^-)=z_1T(\xi_1^-)+\sum\limits_{n=2}^\infty(z_n-z_{n-1})T(\xi_1^-)=T(\xi_1^-),
$$
$$
T_1(\xi_k^+)=z_kT(\xi_k^+)+z_k^\perp T(\xi_k^+)=T(\xi_k^+).
$$
This means that $T=T_1.$

Continuing by this way we obtain
$$
T=\sum\limits_{n=1}^\infty f_n\, \xi_n \otimes \xi_n.
$$
The proof is complete.
\end{proof}

\section{Applications to partial integral equations}

In this section we shall apply the main result of the previous
section to partial integral equations on the space with mixed norm
of measurable functions. For more information about the partial
integral operators and equations can be found in the monograph
\cite{AKZ}.

Let $(S, \Xi, \nu)$ be a measure space and let
$L^{2,\infty}(S\times\Omega)$ be the set of all complex-valued
measurable functions $f$ on $S\times\Omega$ such that
$$
\int\limits_{S}|f(s, \omega)|^2\, d\mu(s) \in L^\infty(\Omega).
$$
Then $L^{2,\infty}(S\times\Omega)$ is a Hilbert--Kaplansky module
over  $L^\infty(\Omega)$ with respect to inner product defined by
$$
\langle f, g\rangle=\int\limits_{S}f(s, \omega)\overline{g(s,
\omega)}\, d\mu(s),\, f,g\in L^{2,\infty}(S\times\Omega).
$$
Let us take  a complex-valued measurable function $k(t, s,
\omega)$ on $S^2\times\Omega$ such that
\begin{equation}\label{HS}
\int\limits_{S}\int\limits_{S}|k(t,s, \omega)|^2\, d\mu(t)\,
d\mu(s)\in L^\infty(\Omega)
\end{equation}
and define an operator $T : L^{2,\infty}(S\times\Omega)\rightarrow
L^{2,\infty}(S\times\Omega)$ by the rule
\begin{equation}\label{part}
 T(f)(t, \omega) =\int\limits_{S}k(t, s, \omega)f(s, \omega)
\,d\mu(s),\, f\in L^{2,\infty}(S\times\Omega).
\end{equation}
For any $\omega\in \Omega$ we put $k_\omega(t, s) = k(t, s,
\omega).$ Then \eqref{HS} means that the function $k_\omega(t, s)$
is a Hilbert--Schmidt kernel for almost all $\omega\in \Omega.$
For almost all $\omega\in\Omega$  the operator $T_\omega :
L^{2}(S)\rightarrow L^{2}(S)$ is defined by
$$
T_\omega(f_\omega)(t) =\int\limits_{S}k_\omega(t, s)f_\omega(s)
\,d\mu(s),\, f_\omega\in L^{2}(S).
$$
Since $T_\omega$ is an integral operator with Hilbert--Schmidt
kernel it follows that  $T_\omega$ is a compact operator for
almost all $\omega\in\Omega.$ For $f\in
L^{2,\infty}(S\times\Omega)$ we have
$$
T(f)(t, \omega) =\int\limits_{S}k(t, s, \omega)f(s, \omega)
\,d\mu(s) =\int\limits_{S}k_\omega(t,
s)f_\omega(s)=T_\omega(f_\omega)(t)
$$
for almost all $(t, \omega)\in S\times\Omega,$ where $f_\omega(s)
= f(s, \omega).$ This means that $\{T_\omega : \omega\in\Omega\}$
is a measurable bundle of compact operators (see \cite{GK01}).
Therefore, by \cite[Theorem 3]{GK01} the operator T is cyclically
compact.

Now suppose that $k(t,s,\omega)=\overline{k(s,t,\omega)}.$ Then
the operator $T$ defined by~\ref{part} is self-adjoint.
Theorem~\ref{MA} implies the following result.

\begin{theorem}\label{PIO}
There are partition  $\{\Omega_0, \Omega_1, \ldots,
\Omega_k,\ldots, \Omega_\infty\}$ of $\Omega$ and orthonormal
families $\{g_{k, n}\}_{n=1}^{k}$  in $\chi_{\Omega_k}
L^{2,\infty}(S\times\Omega)$ and a families $\{\lambda_{k,
n}\}_{n=1}^k$ in $\chi_{\Omega_k} L^\infty(\Omega),$ where $k=1,
\ldots, n,\ldots, \infty,$ such that the following hold:
\begin{enumerate}
\item[(1)] $\chi_{\Omega_0} T=0;$
\item[(2)] $\chi_{\Omega_\infty} |\lambda_{\infty, n}|\downarrow 0;$
\item[(3)] the representation is valid
\begin{equation*}\label{selpio}
T(f)=\chi_{\Omega_\infty}\sum\limits_{n=1}^\infty \lambda_{\infty,
n}\,\langle f, g_{\infty, n}\rangle g_{\infty,
n}+\sum\limits_{k=1}^\infty \chi_{\Omega_k} \sum\limits_{n=1}^k
\lambda_{k, n}\, \langle f, g_{k,n}\rangle g_{k,n}.
\end{equation*}
\end{enumerate}
\end{theorem}

Let us consider the following partial integral equation
\begin{equation}\label{pie}
\int\limits_{S}k(t, s, \omega)f(s, \omega)
\,d\mu(s)=\lambda(\omega) f(t,\omega),
\end{equation}
where $\lambda(\omega)\in L^\infty(\Omega),\, f\in
L^{2,\infty}(S\times\Omega).$

Theorem~\ref{PIO} implies the following condition of solvability
of partial integral equation with self-adjoint kernel.

\begin{corollary}\label{PIO}
Suppose that $k(t,s,\omega)=\overline{k(s,t,\omega)}.$ Then
partial integral equation \eqref{pie} is solvable if and only if
there exists a non zero $\pi\in \nabla$ and $\lambda_{k,n}$ such
that $\pi\lambda=\pi\lambda_{k,n}.$
\end{corollary}

\begin{remark}
Suppose that $(\Omega, \Sigma, \mu)$ is an atomless measure space.
Then the operator $T$ defined by~\eqref{part} is compact if and
only if $k(t,s,\omega)\equiv 0.$ Indeed, suppose that
$k(t,s,\omega)\neq  0.$ Since $T$ is a compact operator, it
follows that there exists an eigenvalue $\lambda$ of $T.$ Note
that the  subspace of eigenvectors corresponding to $\lambda$ is
finite dimensional.

On the other hand, since the operator $T$ is
$L^\infty(\Omega)$-linear, we conclude that the  subspace of
eigenvectors corresponding to $\lambda$ is a non trivial
$L^\infty(\Omega)$-module, in particular, it is an infinite
dimensional, because $L^\infty(\Omega)$ is an infinite
dimensional. This contradiction implies that $k(t,s,\omega)\equiv
0.$
\end{remark}

\section{Self-adjoint compact operators in type I von Neumann algebras}

In this section we shall apply the main result of the  section 2
to compact operators relative von Neumann algebras of type I. For
more information about the compact operators relative to von
Neumann algebras  can be found in  \cite{Kaf}.

An operator $x\in M$ is \textit{compact relative} to $M,$ if it is
the limit in the norm of finite operators in $M,$ i.e. of
operators for which the left support
$$
l(y)=\inf\{p\in P(M): py=y\}
$$
is finite.

\begin{theorem}\label{main}
Let $M$ be a type I von Neumann algebra and let $x$ be a compact
operator relative to  $M.$ If $x$ is self-adjoint then there are a
sequence of mutually orthogonal central projections  $\{z_0, z_1,
\ldots, z_k,\ldots, z_\infty\}$ in $M$ and the families of
mutually orthogonal abelian projections $\{p_{k, n}\}_{n=1}^{k}$
in $z_k M$ and a families of central elements $\{f_{k,
n}\}_{n=1}^k$ in $z_k M,$ where $k=1, \ldots, n,\ldots, \infty,$
such that the following hold:
\begin{enumerate}
\item[(1)] $z_0 x=0;$
\item[(2)] $z_\infty |f_{\infty, n}|\downarrow 0;$
\item[(3)] the representation is valid
\begin{equation}\label{selvn}
x=z_\infty\sum\limits_{n=1}^\infty f_{\infty, n}\, p_{\infty,
n}+\sum\limits_{k=1}^\infty z_k \sum\limits_{n=1}^k f_{k, n}
p_{k,n}.
\end{equation}
\end{enumerate}
\end{theorem}

For the proof of this theorem we need several lemmata.

Consider  a Hilbert space   $H.$   A mapping $s:\Omega\rightarrow
H$ is called \textit{simple,} if
$s(\omega)=\sum\limits_{k=1}^{n}\chi_{A_{k}}(\omega)c_k,$ where
$A_k\in\Sigma, A_i\cap A_j=\emptyset, \,i\neq j,\, \,c_k\in
H,\,k=\overline{1, n},\, n\in\mathbb{N}.$ A mapping
$u:\Omega\rightarrow H$ is said to be \textit{measurable,} if for
each $A\in\sum$ with $\mu(A)<\infty$ there is a sequence $\{s_n\}$
of simple maps such that $\|s_n(\omega)-u(\omega)\|\rightarrow0$
almost everywhere on  $A.$

Denote by $\mathcal{B}(\Omega, H)$ the set of all bounded
measurable mappings from $\Omega$ into $H,$ and let $L^\infty
(\Omega, H)$ denote the space of all equivalence classes  with
respect to the equality almost everywhere. The equivalence class
from $L^\infty(\Omega, H)$ which contains the measurable map
$\xi\in \mathcal{B}(\Omega, H)$ denotes as  $\widehat{\xi}.$ We
shall identify the element $\xi\in \mathcal{B}(\Omega, H)$ and the
class $\widehat{\xi}.$ It is clear that the function  $\omega
\rightarrow \|\xi(\omega)\|$ is measurable for all  $\xi\in
\mathcal{B}(\Omega, H).$ Denote by $\|\widehat{\xi}\|$ the
equivalence class containing the function $\|\xi(\omega)\|.$ The
algebraic operations on $L^\infty(\Omega, H)$ defined by usual
way: $\widehat{\xi}+\widehat{\eta}=\widehat{\xi+\eta},
a\widehat{\xi}=\widehat{a \xi}$ for all $\widehat{\xi},
\widehat{\eta} \in L^\infty(\Omega, H), a\in L^\infty(\Omega).$
Let us consider  $L^\infty(\Omega)$-valued inner product
$$
\langle \xi, \eta \rangle=\langle \xi(\omega), \eta(\omega)
\rangle_{H},
$$
where  $\langle \cdot, \cdot\rangle_{H}$ is the inner product in
$H.$ Then $L^\infty(\Omega, H)$ is a Hilbert~--Kaplansky module
over $L^\infty(\Omega).$

It is known \cite{Alb2} that $\alpha$-dimensional Hilbert space
$H$ the Hilbert -- Kaplansky module $L^{\infty}(\Omega, H)$ is
$\alpha$-homogeneous.

Let   $B(L^{\infty}(\Omega, H))$ be the algebra of all bounded
$L^{\infty}(\Omega)$-linear operators on  $L^{\infty}(\Omega, H).$
Taking into account that $L^{\infty}(\Omega, H)$ is a Hilbert --
Kaplansky module over  $L^{\infty}(\Omega)$ we get that
$B(L^{\infty}(\Omega, H))$ is an $AW^\ast$-algebra of type  I with
the center is $\ast$-isomorphic to  $L^{\infty}(\Omega).$ Suppose
that $\dim H=\alpha.$ Then $L^{\infty}(\Omega, H)$ is
$\alpha$-homogeneous and by \cite[Theorem 7]{Kap}  the algebra
$B(L^{\infty}(\Omega, H))$ has the type  I$_\alpha.$ The center
$Z(B(L^{\infty}(\Omega, H)))$ of this $AW^{\ast}$-algebra
isomorphic with the algebra $L^{\infty}(\Omega)$ which is a von
Neumann algebra,  and thus by \cite[Theorem 2]{Kap2}
$B(L^{\infty}(\Omega, H))$ is also a von Neumann algebra.
Consequently, if  $\dim H=\alpha$ then  $B(L^{\infty}(\Omega, H))$
is a of type I$_\alpha$ von Neumann algebra.

Now let us consider an arbitrary of type I$_{\alpha}$ homogeneous
von Neumann algebra $M$   with the center isomorphic to
$L^{\infty}(\Omega).$ Taking into account that two von Neumann
algebras of the same type I$_{\alpha}$ with isomorphic centers are
mutually $\ast$-isomorphic, we conclude that that the algebra $M$
is $\ast$-isomorphic to the algebra $B(L^{\infty}(\Omega, H)),$
where $\dim H=\alpha.$

A projection $p\in B(L^{\infty}(\Omega, H))$ is called
$\sigma$-finite-generated if $p(L^{\infty}(\Omega, H))$ is a
$\sigma$-finite-generated module.

\begin{lemma}\label{fir}
A projection  $p\in M\cong B(L^{\infty}(\Omega, H))$ is finite if
and only if it is $\sigma$-finitely-generated.
\end{lemma}

\begin{proof}
''if'' part. Let $p\in M$ be a finite projection.

Case 1. Let $p$ be a projection such that $pMp$ is a
$n$-homogeneous von Neumann algebra. Then $p\in pMp\equiv
B(p(L^{\infty}(\Omega, H)))$ is also $n$-homogeneous algebra. This
implies that $p(L^{\infty}(\Omega, H))$ is a $n$-homogeneous
Hilbert--Kaplansky module. This means that $p$ is a
$n$-homogeneous projection. In particular, $p$ is a
$\sigma$-finite-generated projection.

Case 2. Let $p$ be an arbitrary finite projection. Then there
exists a system of mutually orthogonal  central  projections
$(q_{n})_{n\in F}\subset\mathcal{P}(pMp),$ where $F\subseteq
\mathbb{N},$ with $\sum\limits_{n\in F}q_{n}=p$ such that
$q_{n}pMp$ is a homogeneous von Neumann algebra of type I$_n.$ By
case 1 we get that $q_n p$ is $n$-homogeneous for all $n\in F.$
Thus $p=\sum\limits_{n\in F}q_n p$ is $\sigma$-finite-generated.

''only if'' part. Let $p$ be a $\sigma$-finite-generated
projection. Then   there exists a partition of unity
$\{z_n\}_{n\in F}$  in $\nabla,$ where $F\subseteq \mathbb{N},$
such that $z_n p(L^\infty(\Omega, H))$ is a $n$-homogeneous module
over $z_n L^\infty(\Omega)$  for all $n\in F.$ Therefore
$z_npMp\cong B(z_n p(L^\infty(\Omega, H))$ is  a homogeneous von
Neumann algebra of type I$_n$ for all $n\in F.$ Thus
$p=\sum\limits_{n\in F}q_n p$ is finite. The proof is complete.
\end{proof}

\begin{lemma}\label{cyc}
Let  $M\equiv B(L^{\infty}(\Omega, H)).$ If $x\in M$ is a compact
operator relative $M$ then it is cyclically compact.
\end{lemma}

\begin{proof}
Let   $x\in M$ be  a compact operator relative $M.$ By
\cite[Theorem 1.3]{Kaf} for every $n\in \mathbb{N}$ there exists a
projection $p_n\in M$ such that $\|xp_n\|<1/n$ and
$\mathbf{1}-p_n$ is finite. By Lemma~\ref{fir} $\mathbf{1}-p_n$ is
$\sigma$-finite-generated. Therefore $\mathbf{1}-p_n$ is a
cyclically compact operator on $L^\infty(\Omega, H).$ Thus
$x(\mathbf{1}-p_n)$ is also a cyclically compact operator on
$L^\infty(\Omega, H).$ Since
$\|x-x(\mathbf{1}-p_n)\|=\|xp_n\|<1/n$ we obtain that $x$ is also
a cyclically compact operator on $L^\infty(\Omega, H).$ The proof
is complete.
\end{proof}

\textit{Proof of Theorem~\ref{main}}. Firstly, we will consider a
case homogeneous von Neumann algebra. In this case by
Lemma~\ref{cyc} $x$ is a cyclically compact operator. Therefore by
Theorem~\ref{selfd} we can assume that without lost of generality
it has the following form
\begin{equation*}
x=\sum\limits_{k=1}^\infty f_k\, \xi_k \otimes \xi_k.
\end{equation*}
According to  \cite[Lemma 13]{Kap} we obtain that
$p_k=\xi_k\otimes \xi_k$ is an abelian projection for all $k.$ So
\begin{equation*}
x=\sum\limits_{k=1}^\infty f_k\, p_k.
\end{equation*}

Now if $M$ is an arbitrary von Neumann algebra of type I then we
can consider the decomposition of the algebra $M$ to homogeneous
summands and apply the above assertion.  The proof is complete.

\begin{remark} If $M$ is a type I$_n,\, n<\infty,$ von Neumann
algebra which represented as $n\times n$-matrix algebra over its
center, then Theorem~\ref{main} gives us that any self-adjoint
element from $M$ can be represented as diagonal matrix (cf.
\cite{FT}).

\end{remark}

\section*{Acknowledgments}

The authors acknowledge the MOHE grant ERGS13-024-0057 for support.
The second named author (K.K.) also thanks International Islamic
University Malaysia for kind hospitality and providing all
facilities.

\end{document}